\documentclass[12pt, a4]{amsart}
\usepackage{amscd, amsmath, amssymb}
\usepackage{fullpage}

\theoremstyle{plain}
\newtheorem{thm}{Theorem}[section]

\newtheorem{pro}[thm]{Proposition}
\theoremstyle{definition}

\newtheorem{rem}[thm]{Remark}

\numberwithin{equation}{section}

\newcommand{\R}{\mathbb{R}}

\begin{document}

\title[Fixed point theorems and systems of integral equations]{A positive fixed point theorem with applications to systems of Hammerstein integral equations}  

\date{}

\author[A. Cabada]{Alberto Cabada}
\address{Alberto Cabada, Departamento de An\'alise Ma\-te\-m\'a\-ti\-ca, Facultade de Matem\'aticas, 
Universidade de Santiago de Com\-pos\-te\-la, 15782 Santiago de Compostela, Spain}%
\email{alberto.cabada@usc.es}%

\author[J. {\'A}. Cid]{Jos\'e \'Angel Cid}
\address{Jos{\'e} {\'A}ngel Cid, Departamento de Matem\'aticas, Universidade de Vigo, 32004, Pabell\'on 3, Campus de Ourense, Spain}%
\email{angelcid@uvigo.es}%

\author[G. Infante]{Gennaro Infante}
\address{Gennaro Infante, Dipartimento di Matematica e Informatica, Universit\`{a} della
Calabria, 87036 Arcavacata di Rende, Cosenza, Italy}%
\email{gennaro.infante@unical.it}%

\begin{abstract} 
We present new criteria on the existence of fixed points that combine some monotonicity assumptions with the classical fixed point index theory. As an illustrative application, we use our theoretical results to prove the existence of positive solutions for systems of nonlinear Hammerstein integral equations. An example is also presented to show the applicability of our results.
\end{abstract}



\maketitle

\section{Introduction}
In this manuscript we pursue the line of research developed in the recent papers~\cite{cabcid,CCI,jc-df-fm,df-gi-jp-prse,persson} in order to deal with fixed point theorems on cones that mix monotonicity assumptions and conditions in one boundary, instead of imposing conditions on two boundaries as in the celebrated cone compression/expansion fixed point theorem of Krasnosel'ski\u\i{}. In order to do this we employ the well-known monotone iterative method, combined with the classical fixed point index. In Section 2 we prove two results concerning non-decreasing and non-increasing operators in a shell, in presence of an upper or of a lower solution; in Remark~\ref{compthem} we present a comparison with previous results in this direction.

In~\cite{jc-df-fm} Cid and co-authors, in order to show the existence of positive solutions of the fourth-order boundary value problem (BVP)
\begin{gather}
\begin{aligned}\label{bvp-intro}
 u^{(4)}= \lambda g(t) f(u),\ t\in (0,1),\\
u(0)=u(1)=0=u''(0)=u''(1),
\end{aligned}
\end{gather}
where $\lambda>0$, studied the associated Hammerstein integral equation 
\begin{equation}\label{intHE0}
u(t)=\lambda \int_0^1  k(t,s) g(s)f(u(s))\,ds,
\end{equation}
where $k$ is precisely the Green's function associated to the BVP~\eqref{bvp-intro}. Having defined the constant 
$$
\gamma^*=\displaystyle\max_{t\in [0,1]} \int_0^1
k(t,s) g(s)\, ds,
$$
the main result in~\cite{jc-df-fm}, regarding the BVP~\eqref{bvp-intro}, is the following.

\begin{thm}\label{cfm}
\label{tappl} Assume that $\displaystyle\lim_{s\to \infty}
\frac{f(s)}{s}=+\infty$ and there exists $B\in [0,+\infty]$ such
that $f$ is non-decreasing on $[0,B)$. If \[0< \lambda <
\sup_{s\in(0,B)}\frac{s}{\gamma^* f(s)},\] (with the obvious meaning
when $f(s)=0$), then the BVP~\eqref{bvp-intro} has at least a positive
solution.
\end{thm}
Note that the above theorem is valid for a \emph{specific} Green's function. On the other hand the existence of nonnegative solutions for systems of Hammerstein integral equations has been widely studied, see for example~\cite{ag-or-pw, chzh1, chzh2, dunn-wang1, dunn-wang2, df-gi-do, Goodrich1, Goodrich2, 
 jh-rl1, jh-rl2, gipp-nonlin, karejde, lan, lan-lin-lms, lan-lin-na, ya2, yang-zhang} and references therein. 
In Section~\ref{sectionapp} we give an extension of Theorem~\ref{cfm} to the context of systems of Hammerstein integral equations of the type
\begin{gather}
\begin{aligned}\label{ham-sys-intro}
u_1(t)=\lambda_1 \int_{a}^{b}k_1(t,s)g_1(s)f_1(u_1(s),u_2(s))\,ds, \\
u_2(t)= \lambda_2 \int_{a}^{b}k_2(t,s)g_2(s)f_2(u_1(s),u_2(s))\,ds,
\end{aligned}
\end{gather}
providing, under suitable assumptions on the kernels and the nonlinearities, the existence of a positive solution.

In order to show the applicability of our results, we discuss the following system of second-order ODEs, subject to local and nonlocal boundary conditions, that generates two different kernels,
\begin{gather}
\begin{aligned}\label{1syst-intro}
u_1''(t) + \lambda_1 f_1(u_1(t),u_2(t)) &= 0, \ t\in  (0,1), \\
u_2''(t) + \lambda_2 f_2(u_1(t),u_2(t)) &= 0, \ t\in  (0,1),\\
u_1'(0)=0,\ u_1(1)+u_1'(1)&=0,\\
u_2'(0)=0,\ u_2(1)-\xi u_2(\eta)&=0,\ \eta \in (0,1), \ 0 < \xi<1,
\end{aligned}
\end{gather}
computing all the constants that occur in our theory.

\section{Two fixed point theorems in cones}
A subset $K$ of a
real Banach space $X$ is a {\it cone} if it is closed, $K+K\subset K$,
  $\lambda K\subset K$ for all $\lambda\ge 0$ and  $K\cap(-K)=\{\theta\}$. A cone $K$ defines the partial ordering in
$X$ given by 
$$ x\preceq y \quad \mbox{if and only if $y-x\in K$}.$$
 We reserve the symbol
``$\le$" for the usual order on the real line. For $x,\ y \in X$,
with $x\preceq y$, we define the ordered interval \[[x,y]=\{z\in X : x\preceq z\preceq y\}.\]

The cone $K$ is {\it normal} if there exists $d>0$ such that for all $x, y\in X$ with $0\preceq x\preceq y$ then
$\|x\|\le d \|y\| $ .

We denote the closed ball of center $x_0\in X$ and radius $r>0$ as
\[B[x_0,r]=\{x\in X : \|x-x_0\|\le r \},\] and the intersection of the cone with the open ball centered at the origin and radius $r>0$ as
$$K_r=K \cap \{x\in X : \|x\|<r\}.$$

We recall a well known result of the fixed point theory, known as the monotone iterative method (see, for example,~\cite[Theorem 7.A]{zeidler} or~\cite{amann}). 

\begin{thm} \label{mim} Let $N$ be a real Banach space with normal order cone
$K$. Suppose that there exist $\alpha\le \beta$  such that $T\colon [\alpha,\beta]\subset N\to  N$ is a completely continuous monotone non-decreasing operator with $\alpha \le T \alpha $  and $T\beta \le \beta$. Then $T$ has a fixed point and the iterative sequence  $\alpha_{n+1}= T \alpha_n$, with $\alpha_0=\alpha$, converges to the greatest fixed point of $T$ in $[\alpha,\beta]$, and the sequence $\beta_{n+1} = T \beta_n$, with $\beta_0=\beta$, converges to the smallest fixed point of $T$ in $[\alpha,\beta]$. 
\end{thm}

In the next Proposition we recall the main properties of the fixed point index of a completely continuous operator relative to a cone, for more details see~\cite{Amann-rev, guolak}.  
In the sequel the adherence and the boundary of subsets of $K$ are understood to be relative to $K$.

\begin{pro}\label{propindex} Let $D$ be an open bounded set of $X$ with $0\in D_{K}$ and
$\overline{D}_{K}\ne K$, where $D_{K}=D\cap K$. 
Assume that $T:\overline{D}_{K}\to K$ is a completely continuous operator such that
$x\neq Tx$ for $x\in \partial D_{K}$. Then the fixed point index
 $i_{K}(T, D_{K})$ has the following properties:
 
\begin{itemize}

\item[$(i)$] If there exists $e\in K\setminus \{0\}$
such that $x\neq Tx+\lambda e$ for all $x\in \partial D_K$ and all
$\lambda>0$, then $i_{K}(T, D_{K})=0$.

\item[] For example $(i)$ holds if $Tx\not\preceq x$ for $x\in
\partial D_K$.

\item[$(ii)$] If $\|Tx\|\ge \|x\|$ for $x\in
\partial D_K$, then $i_{K}(T, D_{K})=0$.

\item[$(iii)$] If $Tx \neq \lambda x$ for all $x\in
\partial D_K$ and all $\lambda > 1$, then $i_{K}(T, D_{K})=1$.

\item[] For example $(iii)$ holds if either $Tx\not\succeq x$ for $x\in
\partial D_K$ or $\|Tx\|\le \|x\|$ for $x\in
\partial D_K$.

\item[(iv)] Let $D^{1}$ be open in $X$ such that
$\overline{D^{1}}\subset D_K$. If $i_{K}(T, D_{K})=1$ and $i_{K}(T,
D_{K}^{1})=0$, then $T$ has a fixed point in $D_{K}\setminus
\overline{D_{K}^{1}}$. The same holds if 
$i_{K}(T, D_{K})=0$ and $i_{K}(T, D_{K}^{1})=1$.
\end{itemize}
\end{pro}
We state our first result on the existence of non-trivial fixed points.
\begin{thm}\label{thindex} Let $X$ be a real Banach space, $K$ a normal cone
with normal constant $d\ge 1$ and nonempty interior (i.e. solid) and $T\colon
K\to K$ a completely continuous operator. \newline
Assume that
\begin{itemize}
\item[(1)]  there exist $\beta \in K$, with $T\beta\preceq \beta$, and ${R}>0$ such that $B[\beta,{R}]\subset
    K$,
\item[(2)] the map $T$ is non-decreasing in the set
\[ \mathcal{P}=\Bigl\{x\in K : x\preceq \beta \quad \mbox{and} \quad 
\frac{{R}}{d}\le \|x\|
\Bigr\},
\]{}
\item[(3)] there exists a (relative) open bounded set $V\subset K$ such that $i_K(T,V)=0$ and either
$\overline{K_{R}}\subset V$ or $\overline{V}\subset K_{R}$.
\end{itemize}
Then the map $T$ has at least one non-zero fixed point $x_1$ in $K$,
 that
$$ \mbox{either belongs  to $\mathcal{P}$ or belongs to} 
\left\{ \begin{array}{ll} V\setminus \overline{K_{R}},  & \mbox{in case $\overline{K_{R}}\subset V$,} \\ K_{R}\setminus \overline{V},  & \mbox{in case $\overline{V}\subset K_{R}$.}\end{array}\right.
$$
\end{thm}

\begin{proof}  Since $B[\beta,{R}]\subset K$ we have that if $x\in
K$ with $\|x\|={R}$ then  $x\preceq \beta$.

Suppose first that we can choose $\alpha\in K$ with $\|\alpha
\|={R}$ and $T\alpha \succeq \alpha$. Since $\alpha \preceq \beta$ and
due to the normality of the cone $K$ we have that $[\alpha,
\beta]\subset \mathcal{P}$ which implies that $T$ is non-decreasing on
$[\alpha, \beta]$. Then we can apply the Theorem~\ref{mim} to ensure the existence of
a fixed point of $T$ on $[\alpha, \beta]$, which, in
particular, is a non-trivial fixed point.

Now suppose that such $\alpha$ does not exist. Thus 
$Tx \nsucceq x$ for all $x\in K$ with $\|x\|={R}$, which by Proposition~\ref{propindex}, $(iii)$ implies that $i_K(T,K_{R})=1$. Since, by assumption, $i_K(T,V)=0$ we get the existence of  a 
non-trivial fixed point $x_1$ belonging to the set $V\setminus \overline{K_{R}}$ (when $\overline{K_{R}}\subset V$) or to the 
$K_{R}\setminus \overline{V}$  (when $\overline{V}\subset K_{R}$).
\end{proof}

\begin{rem} \label{compthem}
We note that we can use either Proposition~\ref{propindex}, (i) or Proposition~\ref{propindex}, (ii)  in order to check the assumption (3) in Theorem~\ref{thindex}.
 We also stress that $\mathcal{P}$ is contained in the set $\{x\in K : \frac{R}{d}\le \|x\| \le d\,\|\beta \| \}$. Therefore Theorem~\ref{thindex} is a genuine generalization of the previous fixed point theorems obtained in~\cite{cabcid,CCI,jc-df-fm,df-gi-jp-prse}. 
Moreover, we show in the applications that in many cases is useful to apply Theorem~\ref{thindex} with a set $V$ different from $K_{r}$. 
\end{rem}

We observe that, following some ideas introduced in~\cite[Theorem 2.1]{CCI}, it is possible to modify the assumptions of Theorem~\ref{thindex} in order to deal with non-increasing operators. The next result describes precisely this situation.
\begin{thm}\label{thindexdual} Let $X$ be a real Banach space, $K$ a cone
with nonempty interior (i.e. solid) and $T\colon
K\to K$ a completely continuous operator. \newline
Assume that
\begin{itemize}
\item[(1')]  there exist $\alpha \in K$, with $T\alpha\preceq \alpha$, and $0<R<\|\alpha \|$ such that $B[\alpha,{R}]\subset
    K$,
\item[(2')] the map $T$ is non-increasing in the set
\[ \widetilde{\mathcal{P}}=\Bigl\{x\in K :
R\le \|x\|\le  \|\alpha\|
\Bigr\},
\]{}
\item[(3')] there exists a (relative) open bounded set $V\subset K$ such that $i_K(T,V)=1$ and either $\overline{K_{R}}\subset V$ or
 $\overline{V}\subset K_{R}$.
\end{itemize}
Then the map $T$ has at least one non-zero fixed point such that 
$$ \mbox{either belongs  to $\widetilde{\mathcal{P}}$ or belongs to} \left\{ \begin{array}{ll} V\setminus \overline{K_{R}},  & \mbox{in case $\overline{K_{R}}\subset V$,} \\ K_{R}\setminus \overline{V},  & \mbox{in case $\overline{V}\subset K_{R}$.}\end{array}\right.
$$
\end{thm}

\begin{proof} Let $x\in K $ be such that $\|x\|=R$. Then by { (1')} we have that $x\preceq \alpha$ and since $x,\alpha\in \widetilde{\mathcal{P}}$ it follows from { (2')}  that
$$Tx\succeq T\alpha \succeq \alpha \succeq x.$$ 
Now, if for some $x\in \partial K_{R}$ is the case that $Tx\preceq x$ then we are done. If not, $Tx\npreceq x$ for all $x\in \partial K_{R}$ which by Proposition~\ref{propindex} implies that $i_K(T,K_{R})=0$. This result together with  
{(3')} give the existence of a non-zero fixed point with the desired localization property. \end{proof}

\section{An application to a system of Hammerstein integral equations}\label{sectionapp}

We now apply the results of the previous Section in order to prove the existence of positive solutions of the system of integral equations
\begin{gather}
\begin{aligned}\label{ham-sys}
u_1(t)=\lambda_1 \int_{a}^{b}k_1(t,s)g_1(s)f_1(u_1(s),u_2(s))\,ds:=T_1(u_1,u_2)(t), \\
u_2(t)= \lambda_2 \int_{a}^{b}k_2(t,s)g_2(s)f_2(u_1(s),u_2(s))\,ds:=T_2(u_1,u_2)(t),
\end{aligned}
\end{gather}
where we assume the following assumptions:

\begin{itemize}
\item[($H_1$)] $\lambda_i>0$, for $i=1,2$.

\item[($H_2$)] $k_i: [a,b]\times [a,b] \to [0,+\infty)$ is continuous, for $i=1,2$.   

\item[($H_3$)]  $g_i: [a,b] \to [0,+\infty)$ is continuous, $g_i(s)>0$ for all $s\in [a,b]$, for $i=1,2$.

\item[($H_4$)]   $f_i: [0,+\infty)\times [0,+\infty) \to [0,+\infty)$ is continuous, for $i=1,2$.

\item[($H_5$)] There exist continuous functions $\Phi_i: [a,b]\to [0,+\infty)$ and constants $0<c_i<1$, $a\le a_i<b_i \le b$ such that for every $i=1,2$,
$$k_i(t,s)\le \Phi_i(s) \, \mbox{for $t,s\in [a,b]$ and} \, c_i \cdot \Phi_i(s)\le k_i(t,s) \, \mbox{for $t\in [a_i,b_i]$ and $s\in [a,b]$},$$
 and
$$\gamma_{i,*}:=\min_{t\in [a_i,b_i]} \int_{a_i}^{b_i}g_i(s)k_i(t,s)\,ds>0.$$
\end{itemize}

We work in  the space $C[a,b] \times C[a,b]$ endowed with the norm
$$\|(u_1,u_2)\|:=\max\{\|u_1\|_\infty, \|u_2\|_\infty \},$$ where $\|w\|_\infty:=\displaystyle\max_{t\in[a,b]}|w(t)|.$

Set $c=\min\{{c_1},{c_2}\}$ and let us define
$$
\tilde{K_i}:= \{ w \in C[a,b]: w(t)\geq 0 \, \, \mbox{for all $t\in [a,b]$ and} \, \min_{t\in [a_i,b_i]} w(t) \geq c
\|w\|_\infty  \},
$$
 and consider the
cone $K$ in $C[a,b] \times C[a,b]$ defined by
\begin{equation*}
\label{cone}
\begin{array}{c}
 K:=  \{ (u_1,u_2) \in \tilde{K_1} \times \tilde{K_2}  \},
\end{array}
\end{equation*}
which is a normal cone with $d=1$.

Under our assumptions it is routine to check that the integral
operator
\begin{equation*}
T (u_1,u_2)(t)
:=
\left(
 T_1(u_1,u_2)(t),
T_2(u_1,u_2)(t)
\right),
\end{equation*}
leaves $K$ invariant and is completely continuous.

Now we present our main result concerning the existence of positive solutions for the system \eqref{ham-sys}.

\begin{thm}\label{thHam} Assume that the assumptions $(H_1)-(H_5)$ hold and moreover:

\begin{enumerate}

 \item[$(H_6)$] There exist constants $B_1,B_2>0$ such that  for every $i=1,2$, $f_i(\cdot,\cdot)$ is non-decreasing on $[0, B_1]\times [0, B_{2}]$ (that is, if 
 $(u_1,u_2),(v_1,v_2)\in\R^2$ with $0\le u_i\le v_i\le B_i$ for $i=1,2,$ then $f_i(u_1,u_2)\le f_i(v_1,v_2)$ for $i=1,2$). 
 
 \item[$(H_7)$] For every $M>0$ there exists $\rho=\rho(M) >0$ such that, for every $i=1,2$,
$$\inf \Bigl\{ \frac{f_1(u,v)}{ \rho}:\; (u,v)\in [\rho,\rho/c]\times[0, \rho/c]\Bigr\}>M,$$
$$\inf \Bigl\{ \frac{f_2(u,v)}{ \rho}:\; (u,v)\in [0,\rho/c]\times[\rho, \rho/c]\Bigr\}>M.$$

\end{enumerate}

Then the system \eqref{ham-sys} has at least one positive solution in $K$ provided that
\begin{equation}\label{eqlamb} 0<\lambda_i < \sup_{ r_1\in (0,B_1),\\ r_2\in (0,B_2)}\frac{(1-c) r_i}{f_i(r_1,r_2)\gamma_{i}^{*}}, \end{equation}
where
$$\gamma_{i}^{*}:=\max_{t\in [a,b]} \int_{a}^{b}g_i(s)k_i(t,s)ds>0, \quad \mbox{for $i=1,2$.}$$

\end{thm}

\begin{proof} Due to (\ref{eqlamb}) we can fix $\beta_i \in
(0,B_i)$, $i=1,2$, such that
\begin{equation}\label{eqlamb2} \beta_i -\lambda_i \gamma_i^* f_i(\beta_1,\beta_2)>c \beta_i, \quad i=1,2.\end{equation}

On the other hand, for $M>\max\left\{\displaystyle\frac{1}{\lambda_1 \gamma_{1,*}},\displaystyle\frac{1}{\lambda_2 \gamma_{2,*}}\right\}$ let $\rho=\rho(M)>0$ as in $(H_7)$ and fix $R<\min\left\{\displaystyle\frac{1-c}{1+c}\cdot \beta_1, \displaystyle\frac{1-c}{1+c}\cdot \beta_2, \rho\right\}$.
 
Let us check that assumptions of Theorem~\ref{thindex} are satisfied with
$$\beta(t)=(\beta_1,\beta_2)\quad \mbox{for all $t\in [a,b]$},$$
and
$$ V=\{(u_1,u_2) \in K: \min_{t\in [a_1,b_1]}u_1(t)<\rho\ \text{and}\   \min_{t\in [a_2,b_2]}u_2(t)<\rho\}.$$

\noindent {\it Claim 1.- $B[\beta,R]\subset K$ and $T\beta \preceq \beta$.}

Since $\beta$ is constant and $R<\min\left\{\displaystyle\frac{1-c}{1+c}\cdot \beta_1, \displaystyle\frac{1-c}{1+c}\cdot \beta_2\right\}$ a direct computation shows that $B[\beta,R]\subset K$. Now, from (\ref{eqlamb2}) it follows for each $t\in
[a,b]$ and $i=1,2$
\[
[T_i \beta] (t) =  \lambda_i \int _a^b k_i(t,s)
 g_i(s) f_i(\beta_1,\beta_2) ds \le \lambda_i \gamma_i^* f_i(\beta_1,\beta_2)<\beta_i.
\]

Moreover, since $\|\beta_i-T_i\beta\|_{\infty}\le \beta_i$, $i=1,2$, and taking into account (\ref{eqlamb2}) we have for
$t\in [a_i,b_i]$ and $i=1,2,$
 \[\beta_i-[T_i\beta] (t)=\beta_i-\lambda_i \int _a^b k_i(t,s) g_i(s) f_i(\beta_1,\beta_2) ds
 \ge \beta_i-\lambda_i  \gamma_i^* f_i(\beta_1,\beta_2)>
 c_i \beta_i\ge  c_i  \|\beta_i-T_i\beta\|_{\infty}.\]

As a consequence, we have $T\beta \preceq \beta$,  and the claim is proven.
 
\medbreak

\noindent {\it Claim 2.- $T$ is non-decreasing on the set $\{x\in K :x\preceq \beta\}$}.

Let $u=(u_1,u_2), v=(v_1,v_2)\in K$ be such that $0 \le u_i(t)\le v_i(t)\le \beta_i$ for all $t\in [a,b]$ and $i=1,2$.
Since $f$ is non-decreasing in $[0,\beta_1]\times [0,\beta_2]$ we have for all $t\in [a,b]$ and $i=1,2$,
 \[[T_iv](t)-[T_iu](t)=
\lambda_i \int _a^b k_i(t,s)
  g_i(s) [f_i(v(s)) - f_i(u(s)] ds\ge 0.\]
  
Moreover, for all $t\in [a_i,b_i]$, $r\in [0,1]$ and $i=1,2$,
\begin{align*}
 [T_iv](t)-[T_iu](t) & =   \lambda_i \int_a^b k_i(t,s) g_i(s) [f_i(v(s)) - f_i(u(s))] ds\\
   & \ge  \lambda_i  \int_a^b c \Phi_i(s)  g_i(s) [f_i(v(s)) - f_i(u(s)]ds\\
  & \ge  c  \lambda_i \int_a^b k_i(r,s) g_i(s) [f_i(v(s)) - f_i(u(s))] ds\\
  & = c ([T_iv](r)-[T_iu](r)),
 \end{align*}
therefore $\displaystyle\min_{t\in  [a_i,b_i]} ([T_iv](t)-[T_iu](t)) \ge c
\|T_iv-T_iu\|_{\infty}$, $i=1,2$, so $Tu\preceq Tv$, and since $\mathcal{P}\subset \{x\in K :x\preceq \beta\}$, $T$ is also non-decreasing on $\mathcal{P}$.
 \medbreak

\noindent {\it Claim 3.- $\overline{K_{R}} \subset V$ and $i_{K}(T,V)=0$}. 

Firstly, note that since $R<\rho$ then we have $\overline{K_{R}}\subset K_{\rho} \subset V$.

Now let $e(t)\equiv 1$ for $t\in [a,b]$. Then $(e,e)\in K$ and we are going to prove that
\begin{equation*}
(u_1,u_2)\ne T(u_1,u_2)+\mu (e,e)\quad\text{for } (u_1,u_2)\in \partial
V \quad\text{and } \mu \geq 0.
\end{equation*}

If not, there exist $(u_1,u_2)\in \partial V$ and
$\mu \geq 0$ such that $(u_1,u_2)= T(u_1,u_2)+\mu (e,e)$.

Without loss of generality, we can assume that for all $t\in [a_1,b_1]$ we have
$$
\rho\leq u_1(t)\leq {\rho/c},\\\ \min_{t\in [a_1,b_1]} u_1(t)=\rho \\\ \text{and  }\\\ 0\leq u_2(t)\leq {\rho/c}.
$$

Then, for $t\in [a_1,b_1]$, we obtain
\begin{eqnarray*}
u_1(t)&=&\lambda_1 \int_{a}^{b} k_1(t,s)g_1(s)f_1(u_1(s),u_2(s))\,ds
+ \mu e(t)\\
&\geq& \lambda_1 \int_{a_1}^{b_1}k_1(t,s)g_1(s) f_1(u_1(s),u_2(s))\,ds+{\mu}\ge
\lambda_1 M \rho  \gamma_{1,*} +{\mu}>\rho+\mu.
\end{eqnarray*}

Thus, we obtain
$\rho=\displaystyle\min_{t\in [a_1,b_1]}u(t)>
\rho+\mu\geq \rho$,
a contradiction. 

Therefore by Proposition~\ref{propindex} we have that $i_{K}(T,V)=0$ and the proof is finished.
\end{proof}
\begin{rem}
 The following condition, similar to the one given in~\cite{chzh1},  implies $(H_7)$ and it is easier to check.
\begin{itemize}
\item[$(H_7)^*$] For every $i=1,2$, $\displaystyle\lim_{u_i\to +\infty} \frac{f_i(u_1,u_2)}{u_i}=+\infty, \, \, 
 \text{uniformly w.r.t.}\, u_j \in [0,\infty), j\neq i.$
\end{itemize}
\end{rem}
\begin{rem}
In order to deal with negative kernels $k_i(t,s)<0$ we can require conditions $(H_2)$, $(H_3)$ and $(H_5)$ on the absolute value of the kernel such that $|k_i(t,s)|>0$ and conditions $(H_4)$, $(H_6)$ and $(H_7)$ on  $\mathop{\mathrm{sgn}}(k_i)\cdot f_i$.
\end{rem} 
As an illustrative example, we apply our results to the system of ODEs
\begin{gather}
\begin{aligned}\label{1syst}
u_1''(t) + \lambda_1 f_1(u_1(t),u_2(t)) = 0, \ t\in  (0,1), \\
u_2''(t) + \lambda_2 f_2(u_1(t),u_2(t)) = 0, \ t\in  (0,1),%
\end{aligned}
\end{gather}
with the BCs
\begin{gather}
\begin{aligned}\label{1BC}
u_1'(0)=0,\ u_1(1)+u_1'(1)=&0,\\
u_2'(0)=0,\ u_2(1)=\xi u_2(\eta),\ \eta,\xi &\in (0,1). 
\end{aligned}
\end{gather}
To the system \eqref{1syst}-\eqref{1BC} we associate the system of  integral equations
\begin{gather}
\begin{aligned}\label{syst2}
u_1(t)=\lambda_1\int_{0}^{1}k_1(t,s) f_1(u_1(s),u_2(s))\,ds, \\
u_2(t)= \lambda_2\int_{0}^{1}k_2(t,s) f_2(u_1(s),u_2(s))\,ds,%
\end{aligned}
\end{gather}
where the Green's functions are given by
\begin{equation} \label{ker1}
k_1(t,s)=\begin{cases} 2-t, &s\le t, \\ 2-s,&s>t,
\end{cases}
\end{equation}
and 
\begin{equation}\label{ker2}
k_2(t,s)=\dfrac{1}{1-\xi}(1-s)-\begin{cases}
\dfrac{\xi}{1-\xi}(\eta -s), &  s \le \eta\\ \quad 0,&
s>\eta
\end{cases}
 - \begin{cases} t-s, &s\le t, \\ \quad 0,&s>t.
\end{cases}
\end{equation}
The Green's function $k_1$ was studied in~\cite{jw-tmna} were it was shown that we may take (with our notation)
$$\Phi_1(s)= (2-s),\ \gamma_{1}^{*}=\frac{3}{2}.$$
The choice of $[a_1,b_1]=[0,1]$ gives $$c_1=\frac{1}{2},\ \gamma_{1,*}=1.$$

The kernel $k_2$ was extensively studied in~\cite{jw-wilm, jw-tmna} and is more complicated to be dealt with, due to the presence of the nonlocal term in the BCs. In this case we may take
$$
\Phi_2(s)=k_2(0,s)=\begin{cases}\quad  \dfrac{1-s}{1-\xi},
& \text{ if } \eta <s \leq 1,\\
\dfrac{1-s-\xi(\eta-s)}{1-\xi},
& \text { if } 0 \leq s \leq \eta,
\end{cases}
\quad \gamma_{2}^{*}=\dfrac{1-\xi{\eta}^2}{2(1-\xi)}.
$$ 
The choice, as in~\cite{jw-wilm}, of $[a_2,b_2]=[0,b_2],$
where 
$$
b_2=\begin{cases}  \dfrac{1-\xi{\eta}}{2(1-\xi)},
& \text{ if } 1+\xi{\eta} \leq 2\eta,\\
 \dfrac{1}{(2-\xi)},
& \text { if } 1+\xi{\eta}>2\eta,
\end{cases}
$$
leads to
$$c_2=\dfrac{1-\xi{\eta}-(1-\xi)b_2}{1-\xi\eta}, \ \gamma_{2,*}=\begin{cases} \quad b_2^2,
& \text{ if } 1+\xi{\eta} \leq 2\eta,\\
\dfrac{1-2\xi{\eta}^2+\xi^2{\eta}^2}{2(1-\xi)(2-\xi)},
& \text { if } 1+\xi{\eta}>2\eta.
\end{cases}
$$
We now fix, as in~\cite{jw-tmna}, $\eta=1/2, \xi=1/4$. This gives $b_2=4/7$ and
$$ \gamma_{2}^{*}=\frac{5}{8}, c_2=\frac{25}{49}, \gamma_{2,*}=\frac{19}{56}.$$
Furthermore take 
\begin{equation}\label{exf}
f_1(u_1,u_2)=(2+\sin(u_2))u_1^2,\ f_2(u_1,u_2)=(2+\sin(u_1))u_2^2.
\end{equation}
In the case of the nonlinearities \eqref{exf}, we can choose $B_1=B_2=\pi/2$. We observe that condition $(H_7)^*$ holds, we note that $c=\min\{{c_1},{c_2}\}=1/2$ and that 
 $$\sup_{ r_1\in (0,\pi/2),\\ r_2\in (0,\pi/2)}\frac{ r_i}{2f_i(r_1,r_2)\gamma_{i}^{*}}=+\infty,\ \text{for every}\ i.$$ 
As a consequence, by means of Theorem~\ref{thHam}, we obtain a nonzero solution of the system \eqref{1syst}-\eqref{1BC} for every $\lambda_1,\lambda_2\in (0,\infty)$.

\section*{Acknowledgements}
A. Cabada and J. A. Cid were partially supported by Ministerio de Educaci\'on y Ciencia, Spain, and FEDER, Project MTM2010-15314, G. Infante was partially supported by G.N.A.M.P.A. - INdAM (Italy). This paper was partially written during a visit of G. Infante to the 
Departamento de An\'alise Matem\'atica of the Universidade de Santiago de Compostela. G. Infante is grateful to the people of the 
aforementioned Departamento for their kind and warm hospitality.

\end{document}